\documentclass[a4paper,12pt]{amsart}
\usepackage{geometry}
\usepackage[utf8]{inputenc}
\usepackage[english]{babel}
\usepackage{amsmath}
\usepackage{amssymb}
\usepackage{amsthm}
\usepackage{commath}
\usepackage{mathtools}
\usepackage{hyperref}
\usepackage{tikz}
\usepackage{tikz-cd}
\usepackage{fancyhdr}
\usepackage{quiver}
\usepackage{bm}

\theoremstyle{plain}
\newtheorem{thm}{Theorem}[section]
\newtheorem{prop}[thm]{Proposition}
\newtheorem{conj}[thm]{Conjecture}
\newtheorem{lemma}[thm]{Lemma}

\newtheorem{corollary}[thm]{Corollary}
\newtheorem*{thm*}{Theorem}

\theoremstyle{definition}
\newtheorem{defn}[thm]{Definition}
\newtheorem*{defn*}{Definition}

\newtheorem{example}[thm]{\underline{Example}}

\theoremstyle{remark}

\newtheorem{remark}[thm]{\underline{Remark}}

\DeclareMathOperator{\End}{End}
\DeclareMathOperator{\Rad}{Rad}
\DeclareMathOperator{\rk}{rk}
\DeclareMathOperator{\Gr}{Gr}
\DeclareMathOperator{\st}{\; | \;}

\newcommand{\N}{\mathbb{N}}
\renewcommand{\P}{\mathbb{P}}
\newcommand{\Z}{\mathbb{Z}}
\newcommand{\Q}{\mathbb{Q}}
\newcommand{\R}{\mathbb{R}}
\newcommand{\C}{\mathbb{C}}

\newcommand{\mcal}[1]{\mathcal{#1}}

\newcommand\restr[2]{#1 \raisebox{-.5ex}{$|$}_{#2}}

\numberwithin{equation}{section}

\DeclareMathOperator{\Mat}{Mat}
\DeclareMathOperator{\codim}{codim}

\newcommand{\Qbar}{\overline{\Q}}
\newcommand{\diff}{\text{d}}
\newcommand{\Gm}{\mathbb{G}_\mathrm{m}}
\newcommand{\Gcal}{\mathcal{G}}
\newcommand{\Ccal}{\mathcal{C}}
\NewDocumentCommand {\Hvar} {m o} {
    \IfNoValueTF {#2}%
    { H_{\vect{#1}} }%
    {H_{\vect{#1}, #2}}%
}

\newcommand{\vect}[1]{\underline{#1}} 

\title{Singular intersections in split semiabelian varieties}
\author{Francesco Ballini}
\address[Francesco Ballini]{London, UK}
\email{francescoballini1996@gmail.com}

\author{Laura Capuano}
\address[Laura Capuano]{Department of Mathematics and Physics,
         Roma Tre University \\
         Largo San Leonardo Murialdo 1,
         00146 Rome, Italy}
\email{laura.capuano@uniroma3.it}

\author{Nicola Ottolini}
\address[Nicola Ottolini]{Institut de mathématiques, 
    Université de Neuchâtel \\
    Rue Emile-Argand 11,
    2000 Neuchâtel, Switzerland}
\email{nicola.ottolini@unine.ch}
\email{nicolaotto@outlook.it}
\date{}


\begin{document}

\begin{abstract}
    Let $\mathcal G$ be a split semiabelian variety of dimension $n \ge 2$ and let $\mathcal C \subset \mathcal G$ an irreducible curve, where everything is defined over $\Qbar$. We prove that, under the natural hypothesis that $\mathcal C$ is not contained in any proper algebraic subgroup of $\mathcal G$, the set of points $P$ of $\mathcal C$ lying in an algebraic subgroup $H$ which intersects the curve tangentially at $P$ is finite. This generalizes a result for curves in 2-dimensional tori by March\'e and Maurin.  We apply these results to square-freeness of certain geometric divisibility sequences. 
\end{abstract}
\maketitle
\section{Introduction}

One instance of the Zilber-Pink conjecture, proposed independently by Bombieri, Masser and Zannier \cite{BMZ99} and Zilber \cite{Zilber2002} for tori and by Pink \cite{Pink2005} in the more general framework of mixed Shimura varieties, asserts that, if $V$ is an irreducible subvariety of a semiabelian variety $G$ not contained in any proper algebraic subgroup, then the intersection of $V$ with the union of the algebraic subgroups of $G$ of codimension strictly bigger than the dimension of $V$ is not Zariski-dense in $V$. 
The conclusion is motivated by dimension considerations: indeed, if we consider a subvariety $V$ and an algebraic subgroup $H$, we expect that the intersection will be empty if $\dim V+\dim H < \dim G$; 
consequently, if we consider the intersection of $V$ with all the algebraic subgroups satisfying this condition, then we expect that this intersection will be ``small'' in $V$ unless there is a geometric reason for this not to be the case, and in this case we speak of an ``unlikely intersection''. 
The conjecture is still open in full generality; however, for the case of curves, it is known due to works of Bombieri, Masser and Zannier \cite{BMZ99, BMZ08} and Maurin \cite{Maurin08} in the case of tori, of Habegger and Pila \cite{HabeggerPila} and Barroero and Dill \cite{BD22} in the case of abelian varieties, and of Barroero, Kuhne and Schmidt \cite{BKS} in the case of curves in semiabelian varieties, when everything is defined over $\Qbar$ (see also \cite{BHMZ, CMPZ, Viada2003, Viada2008, Remond2003, Galateau2010, BCMOS} for partial results in these contextes). For more references about these problems, see \cite{Zannier2012, Pila2022}.
\medskip

Let us now consider, for an irreducible curve $\mathcal C$ in a semiabelian variety $G$, the intersections between the curve and all the algebraic subgroups of codimension equal to one.
This set of ``just likely intersections'' has been considered in the literature, also in more general cases, usually proving that it is (analytically or Zariski) dense (see, for example, \cite{ACZ20, Gao20a, Gao21, G23, TT23, BKU24, ES25, DG25}).
However, under the more restrictive hypothesis that the curve is not contained in a translate of a proper algebraic subgroup and everything is defined over $\Qbar$, the set is still ``sparse'', i.e. is a set of bounded height (see \cite{BMZ99, Viada2003}).
%

Looking at this set, one can study the geometry of the intersections, i.e. one can look at the multiplicities of intersection between the curve and the subgroups. 
The general expectation, in this case, is that ``for most subgroups'' the intersection between the curve and the subgroup will be transverse. 
This problem has already been considered by March\'e and Maurin \cite{Maurin23} in the case of a curve in a two-dimensional torus, and in the context of one parameter families of abelian varieties by Corvaja, Demeio, Masser, Zannier \cite{CDMZ21}, Ulmer and Urzua \cite{UU20, UU21} and Ottolini \cite{Ottolini}.
The results by Ulmer and Urzua were later extended to positive characteristic, using different techniques, by Ulmer and Voloch \cite{UlmerVoloch25}.
Outside algebraic groups, similar questions have also been studied in the modular setting, where theorems in the same direction have been proven by Spence \cite{Spence2019}, Aslayan \cite{Aslanyan2022} and Ballini \cite{BalliniTesi}.

In this work, we prove that the expectation is actually true when one considers a curve in a split semiabelian variety, and everything is defined over $\Qbar$. 
Our main result is the following.

\begin{thm}\label{thm:split_semiab}
Let $\Gcal$ be a split semiabelian variety of dimension $n\ge 2$ and let $\Ccal \subseteq \Gcal$ be an irreducible curve defined over $\Qbar$ and suppose that $\Ccal$ is not contained in a proper algebraic subgroup.
Consider the set $\Ccal^{[1, sing]}$ of points $P \in \Ccal$ such that there is a proper algebraic subgroup $H$ such that:
\begin{enumerate}
\item $P$ lies in $\Ccal \cap H$;
\item the tangent space $T_P\Ccal$ of $\Ccal$ at $P$ is included in the tangent space $T_P H$ of $H$ at $P$ (namely, the curve $\Ccal$ is tangent to $H$ at $P$).
\end{enumerate}
Then, $\Ccal^{[1, sing]}$ is a finite set.
\end{thm}

Notice that this statement is completely trivial if $n=1$: in this case, $\Ccal=\Gcal$, and the only subgroups of codimension 1 are torsion points.
However, the tangent space of points is trivial, so the set $\Ccal^{[1, sing]}$ is 
empty, as the tangent space to any point of the curve is non-trivial.

The assumption that $\Ccal$ is not contained in a proper algebraic subgroup is 
necessary: if $\Ccal$ were contained in a proper subgroup,
say $H$, then any point $P \in \Ccal = \Ccal \cap H$ would satisfy the two conditions above.

\medskip
This theorem has some consequences also in the study of singular intersections for curves in abelian schemes.
Ottolini \cite[Theorem 1.4]{Ottolini}, using techniques substantially different from those adopted in the present paper, proved the finiteness of the set of singular intersection points between a curve not contained in any proper subgroup scheme and the subgroup schemes of codimension at least 1, under the additional assumption that the abelian scheme does not contain any isotrivial elliptic component.
Theorem \ref{thm:split_semiab}, in the special case where $\Gcal$ is a power of an elliptic curve, allows us to drop this condition.

\medskip

To prove this result, since the statement is invariant under isogenies (see Lemma \ref{lemma:invariant_isogeny}), by Poincar\'e Reducibility Theorem \cite[Corollary A.5.1.8]{Hindry2000}, we can assume that $\mathcal G$ is the product of
a torus and powers of pairwise non-isogenous simple abelian varieties. 
By \cite[Theorem 1.1]{BKS}, the set of points of intersection between the curve $\Ccal$ and the algebraic subgroups of codimension bigger or equal to $2$, without considering the tangency condition, is finite; 
hence, we only need to consider the intersections of $\Ccal$ with algebraic subgroups of condimension 1. 
This in turn allows us to reduce ourselves to the cases in which $\Gcal=G^n$ is a power of either $G=\Gm$ or of an elliptic curve $G$ defined over $\Qbar$. 

Since everything is defined over $\Qbar$, the points that lie in the intersections between the curve and proper algebraic subgroups are algebraic; hence, to prove finiteness, it is enough by Northcott property to prove a bound on the height and on the degree of the points. 
To do this, we first prove Theorem \ref{thm:split_semiab} under the additional condition that the curve $\Ccal$ is not contained in any proper algebraic coset of $\Gcal$. 
Under this assumption, results of Bombieri, Masser and Zannier \cite{BMZ99} for the toric case and Viada \cite{Viada2003} for the elliptic case give the bound for the height of the points of singular intersection between the curve and proper algebraic subgroups.
Moreover, under the same assumption, we are able to get a bound for the degree of the points of singular intersection between the curve and \emph{all proper cosets} (see Proposition \ref{prop:degree}),
using the Gauss map and Bezout's Theorem.
We point out that, for the points of singular intersection between the curve and the proper algebraic subgroups, such a bound on the degree is \emph{independent} of the bound on the height.
This is substantially different from more typical proofs of unlikely intersections type of results, such as the ones mentioned above.
Having a bound for the height and for the degree allows us to conclude.

To prove the general case, we can assume that the curve $\Ccal$ is contained in a proper coset $W$, but not in a proper subgroup.
Moreover, we can assume $W$ to be maximal and, after performing a suitable isogeny, that it is of the form $G^d \times \{ \vect{\xi} \}$, for a tuple $\vect{\xi} \in G^{n-d}$ of independent elements.
By projecting on the first $d$ components, we can turn our starting problem into studying the points of singular intersection between a curve not contained in any proper coset and one-codimensional algebraic cosets of a special shape, which we will call \emph{$\vect{\xi}$-special subvarieties}.
This reduction is analogous to the ones performed in \cite{Remond2003, Pink2005, Maurin08}.
The same bound for the degree from the previous case applies to these points as well, so to prove finiteness, it is enough to get a bound for the height, as the ones from Bombieri, Masser and Zannier and Viada do not apply here.
To prove the desired bound, we employ substantially different techniques from the ones used by Bombieri, Masser and Zannier and Viada.
First, we use the Gauss map and some properties of the height with respect to morphisms between projective varieties to control the logarithmic height of the points of singular intersection between the curve $\Ccal$ and a $\vect{\xi}$-special subvariety $U$ in terms of the logarithmic height of the coefficients of $U$.
A lattice generator bound due to Masser \cite{Masser88} ensures that, for every singular intersection point $P$, we can choose a $U$ with coefficients that are small with respect to the degree of $P$ and the logarithmic height of $P$.
Combining these two bounds with the degree bound leads to the desired height bound, thus concluding the proof.
An alternative argument for the first step, for $\Gm^n$ and elliptic curves without complex multiplication, has also been provided by the first named author in his PhD thesis \cite[\S 5.5]{BalliniTesi} using Puiseux series.

\medskip
We conclude this section by showing, analogously to what observed by Marché and Maurin \cite{Maurin23}, and Ulmer and Urz{\'u}a \cite{UU20, UU21}, how the contents of Theorem \ref{thm:split_semiab}
can be reframed as a statement about unlikely intersections in a different ambient variety.
Keeping the same notation, we will denote by $\Gcal$ a split semiabelian variety of dimension $n$.
Let $T\Gcal$ be the tangent bundle of $\Gcal$, and $\P T\Gcal$ the associated projective bundle;
thus, $\P T\Gcal \to \Gcal$ is a $\P^{n-1}$-bundle
whose fibers parametrize lines in the corresponding tangent space, and the total space $\P T\Gcal$ is a smooth quasi-projective variety of dimension $2 n - 1$.

Given a curve $\Ccal \subseteq \Gcal$, there is a canonical lift of $\Ccal$ to $\widetilde{\Ccal} \subseteq \P T\Gcal$, constructed in the following way.
The inclusion $\iota: \Ccal \to \Gcal$ induces a map $\tilde{\iota}:\Ccal^{sm} \to \P T\Gcal$
from the smooth locus of $\Ccal$, sending a point $P \in \Ccal^{sm}$ to the class of its tangent line $T_P \Ccal \subseteq T_P \Gcal$ in $\P T\Gcal$.
We take $\widetilde{\Ccal}$ to be the closure of the image of this map.
Note that, in this way, $\widetilde{\Ccal}$ is also a curve.

In a similar way, for $W\subseteq\Gcal$ a smooth subvariety we can define $\widetilde{W} \subseteq \P T\Gcal$ as 
\[
\widetilde{W} = \{(P, [\Delta])\st P \in W, \Delta \subseteq T_{P}W \}.
\]
By this construction, $\dim \widetilde{W} = 2 \dim W -1$.

A curve $\Ccal$ and a subvariety $W$ intersect at a point $P$ tangentially if and only if their lifts $\widetilde{\Ccal}$ and $\widetilde{W}$ intersect at a point in $\P T\Gcal$ over $P$.
Now, whenever $\dim W \leq \dim \Gcal - 1$, we have $\dim \widetilde{W} \leq \dim  \P T \Gcal - 2$, so the intersection between $\widetilde{W}$ and $\widetilde{\Ccal}$ is unlikely.

We now recast Theorem \ref{thm:split_semiab} using this new interpretation.
\begin{thm}
\label{thm:split_semiab_reformulated}
Let $\Gcal$ be a split semiabelian variety of dimension $n \ge 2$, and let $\Ccal$ be an irreducible subvariety of $\Gcal$ of dimension 1, not contained in any proper algebraic subgroup, where everything is defined over $\Qbar$.
Let $\widetilde{\Ccal}$ and, for any $W \subseteq \Gcal$ smooth subvariety, $\widetilde{W}$ be as above.
Then, the set
\[
\widetilde{\Ccal}^{[2]}:= \bigcup_{\substack{H < \Gcal\\ \dim H \leq \dim \Gcal - 1}} \widetilde{\Ccal} \cap \widetilde{H} \subseteq \P T\Gcal,
\]
where $H$ varies through the proper algebraic subgroups of $\Gcal$, is a finite set.
\end{thm}

This reformulation lends itself to some generalizations, which could be interesting to study in the future.

First of all, one may ask whether Theorem \ref{thm:split_semiab} holds for more general semiabelian varieties.
The structure of subgroups here is more complicated compared to the abelian and toric settings, so some of the arguments of this paper do not apply directly.
It should be however possible to complement those with the arguments of \cite{BKS} to recover the claim.

The formulation as unlikely intersections suggests that a generalization of Theorem \ref{thm:split_semiab_reformulated} for higher dimensional subvarieties may hold, analogue to the Zilber-Pink conjecture in Pink's formulation (see \cite{Pink2005}).
\begin{conj}
\label{conj:sing_inters_const_higher_dim}
Let $\Gcal$ be a semiabelian variety, and let $V$ be an irreducible smooth subvariety of $\Gcal$, not contained in any proper algebraic subgroup, where everything is defined over $\C$.
Let, for any $W\subseteq\Gcal$ subvariety, $\widetilde{W}$ be the lifting of $W$ to $\P T\Gcal$ and let
\[
\widetilde{V}^{[2 \dim V]}:= \bigcup_{\substack{H < \Gcal\\ \dim H \leq \dim \Gcal - \dim V}} \widetilde{V} \cap \widetilde{H} \subseteq \P T\Gcal;
\]
then, the projection of $\widetilde{V}^{[2 \dim V]}$ onto $V$ is non-Zariski dense in $V$.
\end{conj}

We remark that in our proof of Theorem \ref{thm:split_semiab} it is crucial (in two distinct places) that we know that the Zilber-Pink conjecture is true for curves in semiabelian varieties defined over $\Qbar$.
Therefore, we expect that proving Conjecture \ref{conj:sing_inters_const_higher_dim} should be at least as difficult as proving the Zilber-Pink conjecture.
On the other hand, one could hope to have a conditional proof of Conjecture \ref{conj:sing_inters_const_higher_dim}.

\subsection*{AI statement}
No AI has been used in the writing of this paper, other than for grammatical and spell checking reasons.

\section{Reduction steps}
In this section, we reduce the proof of Theorem \ref{thm:split_semiab} to the case of powers of an elliptic curve or split algebraic tori.
To do so, we start by showing that our problem is invariant under isogenies.
\begin{lemma}
\label{lemma:invariant_isogeny}
Let $\Gcal$ and $\Gcal'$ be split semiabelian varieties, and let $f: \Gcal' \to \Gcal$ be an isogeny.
Moreover, let $\Ccal \subseteq \Gcal$ be a curve satisfying the assumptions of Theorem \ref{thm:split_semiab}.
Then, if the claim of Theorem \ref{thm:split_semiab} holds for all irreducible components of $f^{-1}(\Ccal)$, it holds for $\Ccal$.
\end{lemma}
\begin{proof}
The proof follows from the fact that every isogeny over fields of charachteristic 0 is (analytically) a local isomorphism.
Hence, they preserve the tangency conditions required by the statement of Theorem \ref{thm:split_semiab}.
\end{proof}

Using Poincaré Reducibility Theorem (see \cite[Corollary A.5.1.8]{Hindry2000}), we can then assume that the split semiabelian variety $\Gcal$ of Theorem \ref{thm:split_semiab} is of the form 
\[
\Gcal = \Gm^{k} \times \prod_{i=1}^l A_i^{n_i},
\]
where the $A_i$ are pairwise non-isogenous, simple abelian varieties.
For ease of notation, we will adopt the index $0$ for the toric component, and we will write $A_0^{n_0}$ instead of $\Gm^k$.
We recall that (see, for example, \cite[Section 5]{MW_85}) all the algebraic subgroups of this split semiabelian variety $\mathcal G$ are contained in subgroups of the same dimension and of the form
\[
H=H_0 \times H_1 \times \dots \times H_l,
\]
where $H_0$ is a subgroup of $A_0^{n_0}=\Gm^k$ and $H_i$ is a subgroup of $A_i^{n_i}$.
Moreover, all subgroups of codimension at least one will be contained in a subgroup of this form where $H_i=A_i^{n_i}$ for all but one index $i$.

By Theorem 1.1 of \cite{BKS}, the points of intersection between the curve $\Ccal$ and the algebraic subgroups of codimension bigger or equal to 2, without considering the tangency condition, are finite.
Hence, to prove Theorem \ref{thm:split_semiab}, we only need to intersect the curve $\Ccal$ with the subgroups of codimension one, so we can assume that they are of the form
\begin{equation}
\label{eqn:max_subgroup}
H=A_0 ^{n_0} \times \dots \times A_{i_\star-1}^{n_{i_\star-1}} \times H_{i_\star} \times A_{i_\star + 1}^{n_{i_\star+1}} \times \dots \times A_l^{n_l},
\end{equation}
where $H_{i_\star} < A_{i_\star}^{n_{i_\star}}$ is a subgroup of codimension $1$.
This immediately implies that $A_{i_\star}$ is either $\Gm$ (i.e. $i_\star = 0$) or an elliptic curve, and that $n_{i_\star}$ is at least $1$, as in every factor $A_i^{n_i}$ with $n_i$ at least 1 the minimal possible codimension of a subgroup is the dimension of the simple abelian variety $A_i$.
To prove Theorem \ref{thm:split_semiab} it is enough to show that, for every $i_\star$, the set of points of singular intersection between $\Ccal$ and subgroups as in \eqref{eqn:max_subgroup} is finite.

Suppose from now on that $i_\star$ is fixed and that $A_{i_\star}$ is either $\Gm$ (i.e. $i_\star = 0$) or an elliptic curve.
Let $\pi:\Gcal \to A_{i_\star}^{n_{i_\star}}$ be the projection map and consider the image $\pi(\Ccal) \subseteq A_{i_\star}^{n_{i_\star}}$.

It may happen that $\pi(\Ccal)$ is just a point. 
If this happens, then such a point cannot be contained in a proper subgroup of $A_{i_\star}^{n_{i_\star}}$, since otherwise the curve $\Ccal$ would be contained in a proper subgroup of $\Gcal$, contradicting the assumptions of Theorem \ref{thm:split_semiab}.
On the other hand, if $P$ is a point of intersection between $\Ccal$ and a subgroup as in \eqref{eqn:max_subgroup}, then $\pi(P) \in H_{i_\star}$.
Therefore, we conclude that, under the assumption of Theorem \ref{thm:split_semiab}, if $\pi(\Ccal)$ is a point, then $\Ccal$ does not intersect subgroups as in \eqref{eqn:max_subgroup}.

Suppose then from now on that $\pi(\Ccal)$ is a curve and denote by $\tilde{\pi}$ the restriction of $\pi$ to $\Ccal$.
Using an argument similar to the previous paragraph, under the assumptions of Theorem \ref{thm:split_semiab}, the curve $\pi(\Ccal)$ is not contained in any proper subgroup of $A_{i_\star}^{n_{i_\star}}$.
Now, let $H$ be as in \eqref{eqn:max_subgroup} and let $P \in \Ccal \cap H$ be a singular intersection point. 
We will show that such points are finitely many.
Since $\tilde{\pi}$ is a finite map between curves, it is enough to show that there are only finitely many possible values for $\pi(P)$.
It may happen that $P$ is a singular point for $\mathcal{C}$, or that $\pi(P)$ is a singular point for $\pi(\mathcal{C})$.
Since these points are finitelly many, we can (and will) assume that $P$ is a smooth point of $\mathcal{C}$ and that $\pi(P)$ is a smooth point of $\pi(\mathcal{C})$.

Notice that $H_{i_\star}=\pi(H)$ and $\pi(P) \in \pi(\Ccal) \cap H_{i_\star}$.
We have two cases:
\begin{enumerate}
\item The differential $d \tilde{\pi}_P$ at $P$ is the trivial map.
This is the same as saying that $P$ is a ramification point for the map $\tilde{\pi}$.
Since this is a non-constant map between curves, it admits only finitely many ramification points.
\item The differential $d \tilde{\pi}_P$ at $P$ is an isomorphism.
Hence, the image of $d \tilde{\pi}_P$, that is, $d \pi_P(T_P\Ccal)$, is exactly equal to the tangent space of $\pi(\Ccal)$ at $\pi(P)$.
Moreover, since by assumption $T_P \Ccal \subseteq T_PH$, we have that $d\pi_P(T_P \Ccal ) \subseteq d\pi_P(T_PH)= T_P H_{i_\star}$.
In other terms, $\pi(P)$ is a point of singular intersection between $\pi(\Ccal)$ and $H_{i_\star}$.
The finiteness of these points is ensured by 
the following result.
\end{enumerate}

\begin{thm}\label{thm:main}
Let $G$ be $\Gm$ or $E$ for some elliptic curve $E$ defined over $\Qbar$.
Let $n\ge 2$ be an integer, $\Gcal = G^n$, and let $\Ccal \subseteq \Gcal$ be a curve defined over $\Qbar$ and suppose that $\Ccal$ is not contained in a proper algebraic subgroup.
Consider the set $\Ccal^{[1, sing]}$ of points $P \in \Ccal$ such that there is a one-codimensional algebraic subgroup $H$ such that:
\begin{enumerate}
\item $P$ is in $\Ccal \cap H$;
\item the tangent space $T_P\Ccal$ of $\Ccal$ at $P$ is included in the tangent space $T_PH$ of $H$ at $P$ (namely, the curve $\Ccal$ is tangent to $H$ at $P$).
\end{enumerate}
Then, $\Ccal^{[1, sing]}$ is a finite set.
\end{thm}

\section{Structure of subgroups and cosets}  
\label{sec:structure}
\subsection{Multiplicative case}
Let $G = \Gm$ and let $n$ be a fixed positive integer. 
We want to characterize the groups and the cosets of $\Gm^n$.
This topic is fairly classical, and can be found, for example, in \cite[Section 4.2.3]{Zannier2014}.
We recall it here mostly to establish some notation and definitions.

Since $\End(\Gm)=\Z$, the group homomorphisms $\varphi: \Gm^n \to \Gm^d$ correspond to matrices in $\Mat_{d\times n}(\Z)$, as follows: given a matrix $A=(a_{i, j})$, it induces a homomorphism
\[
\varphi_A : (x_1, \dots, x_n) \mapsto \Big( (x_1^{a_{1, 1}} \cdots x_n^{a_{1, n}}), \dots, (x_1^{a_{d, 1}} \cdots x_n^{a_{d, n}}) \Big).
\]
It is easy to check that this correspondence is compatible with compositions, that is, if $A \in \Mat_{m\times n}(\Z)$ and $B \in \Mat_{d\times m}(\Z)$ then $\varphi_B \circ \varphi_A = \varphi_{BA}$.
The following proposition allows us to give a full classification, in terms of these homomorphisms, of the algebraic subgroups of $\Gm^n$ and of the algebraic cosets, that is, the transates of algebraic subgroups.
\begin{prop} \
\begin{enumerate}
\item All the algebraic subgroups of $\Gm^n$ of codimension $d$ are of the form $\ker(\varphi_A)$ for some $A \in \Mat_{d\times n}(\Z)$ of maximal rank.
\item Let $H \subseteq \Gm^n$ be an irreducible component of an algebraic subgroup. Then, $H=\omega T$, where $T$ is an irreducible algebraic subgroup of $\Gm^n$ and $\omega$ is a torsion point.
\item Let $W \subseteq \Gm^n$ be an irreducible component of an algebraic coset, i.e. a translate of an algebraic subgroup. Then, $W = c T$,  where $T$ is an irreducible algebraic subgroup and $c$ is any point in $\Gm^n$.
\end{enumerate}
\end{prop}
This proposition motivates the following definitions.
\begin{defn}
We say that $H \subseteq \Gm^n$ is a \emph{special subvariety} or \emph{torsion coset} if it is an irreducible component of an algebraic subgroup of $\Gm^n$;
we say that $W \subseteq \Gm^n$ is a \emph{weakly special subvariety} if it is an irreducible component of an algebraic coset.
Note that weakly special subvarieties are exactly the irreducible algebraic cosets.
\end{defn}

Suppose that $H$ is a special subvariety of $\Gm^n$ of codimension $d$. 
By the previous proposition, $H$ is an irreducible component of a subgroup $\tilde{H}$, which in turn can be written as $\ker \varphi_A$ for some matrix $A \in \Mat_{d \times n}(\Z)$ of maximal rank.
By performing a base change in both the starting and the arrival spaces, we can turn $A$ into a matrix of the form $(0 \mid B)$, where $B$ is a $d\times d$ diagonal matrix.
In this basis, $\tilde{H}$ is then cut by the equations ${x_{n-d+1}}^{b_{n-d+1}} = 1; {x_{n-d+2}}^{b_{n-d+2}} = 1; \dots; {x_n}^{b_n}=1$.
So, $H$ is described by
\[
\begin{cases}
x_{n-d+1} &= \omega_1\\
&\vdots\\
x_{n} &= \omega_d
\end{cases}
\]
where $\omega_i$ is a $b_{n-d+i}$-th root of unity, for $i=1, \dots, d$.

When $d = 1$, i.e. in the case of hypersurfaces, $A \in \Mat_{1\times n}(\Z)$ is just a row of integers; 
we will denote them by $\vect{a} = (a_1, \dots, a_n) \in \Z^n$.
We will adopt for the rest of this paper the following notation:
\begin{equation*}
\Hvar{a}[c] := \varphi_{\vect{a}} ^{-1} (\{c\}),
\end{equation*}
whenever $\vect{a}$ is a vector of integers and $c \in \Gm$.
If $\vect{a}$ is primitive, that is, the gcd of its elements is $1$, then $\Hvar{a}[c]$ is irreducible, and therefore it is a weakly special subvariety.


\subsection{Elliptic case}
Let $E$ be an elliptic curve defined over $\Qbar$ and let $n$ be a fixed positive integer. 
We want to study the algebraic subgroups and cosets of $E^n$.
This section is also fairly classical, and most of the results can be found, for example, in \cite[Section 3.3]{Masser1983} and \cite{Viada2003}.

Recall that $\End(E)$ is either $\Z$ or (isomorphic to) an order in a quadratic imaginary field.
Similarly to the $\Gm$ case, we have a correspondence between group homomorphisms $\varphi:E^n \to E^d$ and matrices in $\Mat_{d \times n}(\End(E))$, compatible with compositions, which we will also denote by $\varphi$.

\begin{prop}\
\label{prop:elliptic_subgroups_structure}
\begin{enumerate}
\item All algebraic subgroups of $E^n$ of codimension $d$ are \emph{contained} in subgroups of the same codimension of the form $\ker(\varphi_A)$ for some $A \in \Mat_{d \times n}(\End(E))$ of maximal rank.
\item Let $H \subseteq E^n$ be an irreducible component of an algebraic subgroup. 
Then, $H=\omega + A$, where $A$ is an irreducible abelian subvariety and $\omega$ is a torsion point.
\item Let $W \subseteq E^n$ be an irreducible component of an algebraic coset, i.e. a translate of an algebraic subgroup. 
Then, $W = c + A$,  where $A$ is an irreducible abelian subvariety and $c$ is any point in $E^n$.
\end{enumerate}
\end{prop}

\noindent Similarly to the $\Gm^n$ case, we introduce the following definitions.
\begin{defn}
We say that $H \subseteq E^n$ is a \emph{special subvariety} or \emph{torsion coset} if it is an irreducible component of algebraic subgroup of $E^n$;
we say that $W \subseteq E^n$ is a \emph{weakly special subvariety} if it is an irreducible component of an algebraic coset.
Note again that weakly special subvarieties are exactly the irreducible algebraic cosets.
\end{defn}

Recall that any abelian subvariety $A < E^n$ is contained in a codimension one algebraic subgroup defined by $a_1 P_1 + \cdots + a_n P_n = O$, where the $P_i$'s denote the projections on the $i$-th factor, $O$ is the zero element of $E$, and the $a_i$'s are elements of the endomorphism ring of $E$.
We will adopt the following notation: we write $\Hvar{a}[c]$ for the subvariety of $E^n$ cut by
\begin{equation*}
a_1 P_1 + \cdots + a_n P_n = c,
\end{equation*}
for $\vect{a} = (a_1, \ldots, a_n)\in (\End(E))^n$ and $c \in E$. 
In particular, every coset of $E^n$ will be contained in a coset of this shape.
Unlike the $\Gm$ case, we can only guarantee containment, as some irreducible algebraic subgroups of $E^n$ may not be expressible in the form $\ker(\varphi_A)$ with $A$ of maximal rank in those cases where $\End(E)$ is not a PID.
Notice that, as a consequence of Proposition \ref{prop:elliptic_subgroups_structure}, it is always true that every irreducible algebraic subgroup is an irreducible component of a subgroup of the form $\ker(\varphi_A)$ with $A \in \Mat_{d \times n}(\End(E))$ of maximal rank.
In particular, they will be the irreducible component containing the identity.
For example, let $E_5$ be a CM elliptic curve with $\End(E_5)=\Z[\sqrt{-5}]$, and consider in $E_5^2$ the subgroup $H$ defined by
\[
\left\{
\begin{aligned}
2 P_1 + (1 + \sqrt{-5}) P_2 &= 0\\
(1 - \sqrt{-5}) P_1 + 3 P_2 &= 0
\end{aligned}
\right.
.
\]
Then, $H$ is an irreducible subgroup of codimension one that cannot be described by a single equation. 
Moreover, it is isogenous but not isomorphic to $E_5$.\\

We notice that, for any elliptic curve $E$, if we identify the elements of $\End(E)$ with the elements of $\Z$ or with the elements of the appropriate order in a quadratic field in the CM case, for every $n$-tuple $\vect{p}=(p_1, \dots, p_n) \in (\End(E))^n$ the abelian subvariety given by the connected component of the identity of $\Hvar{p}=\Hvar{p}[O]$ is precisely the image via the exponential map of the linear subspace cut by the equation $p_1 z_1 + \cdots + p_n z_n = 0$ in the tangent space of $E^n$ at the origin.\\

Let now $A<E^n$ be an (irreducible) abelian subvariety of any dimension. 
By Poincaré Reducibility Theorem (see for example \cite[Corollary A.5.1.8]{Hindry2000}), there exist an abelian variety $A'$ and an isogeny $\varphi:A' \times A \rightarrow E^n$, both defined over $\Qbar$.
By unique factorization, we have that $A'$ is isogenous to $E^d$ and $A$ is isogenous to $E^{(n-d)}$, where $d = \codim A$.
Hence, if $A < E^n$ is an irreducible abelian subvariety, up to an isogeny of $E^n$, we can assume that it is the kernel of the first $d$ coordinates.

Similarly, if $H\subseteq E^n$ is a special subvariety, we can assume, up to isogeny, that
\[
H=\underbrace{E \times \cdots \times E}_{n-d} \times \{\omega_1\}\times\cdots \times \{\omega_{d}\},
\]
with $\omega_1, \dots, \omega_{d}$ torsion elements of $E$ and $d = \codim H$,
and, if $W\subseteq E^n$ is a weakly special subvariety, we can assume, again up to isogeny, that
\[
W=\underbrace{E \times \cdots \times E}_{n-d} \times \{c_1\}\times\cdots \times \{c_{d}\},
\]
with $c_1, \dots, c_d$ points in $E^d(\Qbar)$.

Note that this is not in contrast with the existence of subgroups that are not cut by system of equations of maximal rank: indeed, to write subgroups in that form we need first to perform an isogeny \emph{which depends on the subgroup}.

\section{The Gauss Map}
\label{sec:gauss}
In this section, we will introduce a technical tool, the Gauss map, that will be crucial in later proofs.
Although we will only need it for curves in $\Gm^n$ or in powers of elliptic curves, we introduce it in a much more general context, that of Lie groups.

The term ``Gauss map'' is used in general to denote a map associating to every point of a subvariety $V$ of an ambient space $X$ the tangent space $T_P V$.
It was first considered by Gauss for orientable surfaces in $\R^3$, mapping to the points of the unit sphere corresponding to the normal directions.
He used such a map to study the curvature of surfaces, and ultimately prove the Theorema Egregium in 1827.

Let $\mcal{L}$ be a Lie group. The group structure allows us to identify the tangent spaces at each point using the differential of the translation maps.
Hence, we get a trivialization of the tangent bundle $T\mcal{L} \cong \mcal{L} \times T_e \mcal{L}$, where $T_e \mcal{L}$ is the tangent space at the identity.
Suppose that $V \subseteq \mcal{L}$ is a subvariety, and let $V^{sm}$ be the open dense locus of smooth points.
Then, for each $P \in V^{sm}$ the tangent space $T_PV$ is a well-defined subspace of $T_P \mcal{L}$ of dimension equal to the dimension of $V$. 
Using the above-mentioned trivialization of the tangent bundle of $\mcal{L}$, we can also identify $T_P V$ with a subspace of $T_e \mcal{L}$.
Hence we get a map
\[
\gamma_{\mcal{L}, V} : V^{sm} \to \Gr(\dim V, T_e \mcal{L}),
\]
where $\Gr(\dim V, T_e \mcal{L})$ is the Grassmanian of $(\dim V)$-dimensional vector subspaces of $T_e \mcal{L}$.
We will often suppress the dependence on the Lie group $\mcal{L}$, as such datum will be clear from the context.
This map has been studied previously, see for example \cite[Section 4]{Griffiths1979}, and is a key ingredient in Raynaud's proof of the Manin-Mumford Conjecture \cite{Raynaud83}.

In the case when $V=\Ccal$ is a curve, the Grassmanian $\Gr(\dim V, T_e \mcal{L})$ is just the projective space $\P^{\dim \mcal{L}-1}$.
\begin{remark}
By definition, we have $\gamma_\Ccal (P) = \gamma_{\tau^{-1}_P(\Ccal)} (e) = [T_e (\tau^{-1}_P(\Ccal))]$, where $\tau_P$ is the left translation by $P$.
\end{remark}
Whenever, if $V= \Ccal$ is a curve and $\mcal{L}$ is the analytification of an abelian variety or of $\Gm^n$, then the Gauss map defined above is actually an algebraic map.
For abelian varieties, this was proven by Raynaud in \cite{Raynaud83}; for $\Gm^n$ we check it by explicitly describing the map.
\medskip

Let $\mcal{L}=(\Gm^n)^{an}$. The Gauss map can be written more explicitly as follows. 
Let $x_1 ,\dots, x_n$ denote the coordinates of $\Gm^n$; we also see these as coordinate functions on $\Ccal$ and we have the corresponding differentials $\diff{x_1},\ldots,\diff{x_n}$.
Endowing $T_1\Gm^n$ with the basis $(\frac{\partial}{\partial x_1}, \dots, \frac{\partial}{\partial x_n})$, we then have the map $\gamma_\Ccal : \Ccal^{sm} \to  \Gr(1, T_1 \Gm^n)= \P^{n-1}$ defined by
\[
\gamma_\Ccal: P=(x_1, \dots, x_n) \mapsto \left[ \frac{\diff{x_1} (t_P)}{x_1}: \dots : \frac{\diff{x_n}(t_P)}{x_n} \right],
\]
where $t_P$ is any non-zero element in $T_P \Ccal$. Since all non-zero elements of $T_P \Ccal$ for $P$ a smooth point of $\Ccal$ are proportional, this is well defined.
We notice that this map is just the logarithmic derivative, and therefore we expect multiplicative relations to translate into linear relations.
We will prove such a statement, again in a more general setting.

\begin{prop}
Let $\mcal{L}$ be an abelian Lie group, $\mathfrak{l} = T_e \mcal{L}$ the tangent space at the identity, and $\exp: \mathfrak{l} \to \mcal{L}$ the associated exponential map. 
Let $V$ be a linear subspace of $\mathfrak{l}$, let $\Ccal \subseteq \mcal{L}$ be a curve, and let $H=\exp(V)$.
Then, for every $P \in \Ccal^{sm}$ we have that $\Ccal$ is tangent to $P + H$ at $P$ if and only if $\gamma_\Ccal(P) \in V$.
\end{prop}
\begin{proof}
By definition, $\Ccal$ is tangent to $P + H$ in $P$ if and only if $T_P \Ccal \subseteq T_P (P + H)$, when seen as subspaces of $T_P \mcal{L}$.
Since $\diff{\tau_P}$ is an isomorphism, this is the same as $T_e( \tau^{-1}_P( \Ccal)) \subseteq T_e( \tau^{-1}_P(P + H)) = T_e H = V$.
The claim follows from the fact that $[T_e( \tau^{-1}_P( \Ccal))] = \gamma_P (\Ccal)$.

\end{proof}

In the special case in which $\mcal{L}=\mcal{G}$ is a power of $\Gm^n$ or of an elliptic curve, we get the following.
\begin{corollary}
\label{coroll:tangentplane}
Let $\Ccal$ be a curve in $\Gcal=\Gm^n$ or $\Gcal=E^n$ for some elliptic curve $E$, all defined over a number field $K$, 
and let $W=\Hvar{p}[c]$ be a one-codimensional coset.
Let $P$ be a point of $\Ccal \cap W$ and assume that $\Ccal$ is tangent to $W$ at $P$. 
Then, $\gamma_\Ccal (P)$ is contained in the hyperplane $L \subseteq \P^{n-1}= \P T_e \Gcal$ given by the equation:

\[
p_1 z_1+\cdots + p_n z_n = 0,
\]
where $z_i=\frac{\partial}{\partial x_i}$ are the (usual) coordinates of $\P T_e \Gcal$.
\end{corollary}

\section{Proof for \texorpdfstring{$\Ccal$}{C} not contained in a proper coset}
\label{sec:no_coset}
We recall the notations from Theorem \ref{thm:main}. 
Let $G$ be $\Gm$ or $E$ for some elliptic curve $E$ defined over $\Qbar$,
let $n$ be a positive integer, $\Gcal = G^n$ and let $\Ccal \subseteq \Gcal$ be an irreducible curve defined over $\Qbar$ that is not contained in any proper algebraic subgroup.
Let us fix a number field $K$ over which both $\Gcal$ and $\Ccal$ are defined.

Let us further assume in this section that $\Ccal$ is not contained in any proper algebraic coset of $\Gcal$.
Under this assumption, we will show a bound for the degree of the field of definition of a point $P \in \Ccal \cap W$, where $W$ is a one-codimensional weakly special subvariety of $\Gcal$ and $\Ccal$ is tangent to $P$ at $W$;
this bound will be independent of $W$. 
The statement of the theorem will follow combining this degree bound with height bounds for points in a curve not contained in a weakly special subaviety lying in a proper algebraic subgroup proven by Bombieri, Masser and Zannier \cite{BMZ99} in the case of tori and by Viada \cite{Viada2003} in the ellptic curve case.

We remark that this height bound does not hold anymore if the curve is contained in a weakly special subvariety, as shown in \cite{BMZ99} and \cite{Viada2003}; therefore, this 
last hypothesis is crucial to perform the argument.

\begin{prop}\label{prop:degree}
Let $G=\Gm$ or $E$ for some elliptic curve $E$ defined over a number field $K$.
Moreover, if $E$ has complex multiplication, after identifying $\End(E)$ with an appropriate order in a quadratic field, suppose also that $\End(E) \subseteq K$.
Let $n$ be a positive integer, let $\Gcal = G^n$,
let $\Ccal \subseteq \Gcal$ be a curve, also defined over $K$, and assume that is not contained in any proper algebraic coset and that the associated Gauss map $\gamma_{\mathcal{C}}$ is also defined over $K$.  
There is a constant $\delta_1 > 0$, depending only on $\Ccal$ (and $\Gcal$), such that,
for every one-codimensional coset $W$ defined over $\overline{K}$ and every point $P \in \Ccal \cap W$ with $\Ccal$ tangent to
$W$ at $P$:
\[
[K(P ) : K] \leq \delta_1.
\]
\end{prop}

\begin{proof}
Up to removing a finite number of points, we may assume $\Ccal$ to be a smooth curve, so that $\gamma_\Ccal$ is defined everywhere.
Up to enlarging $W$, we may assume that $W=\Hvar{p}[c]$ for some $\vect{p} \in \End(G)^n$ and $c \in G(\overline{K})$.
We will denote by $L$ the hyperplane cut by $p_1 z_1+\cdots + p_n z_n = 0$ in $\P T_0 \Gcal = \P^{n-1}$.

Notice that the image $\gamma_\Ccal (\Ccal)$ cannot be contained in $L$.
Indeed, consider the map $\varphi_{\vect{p}}:\Gcal \to G$ given by
\[
\varphi_{\vect{p}}: (x_1, \dots, x_n) \mapsto x_1^{p_1} \cdots x_n^{p_n}
\]
in the toric case
(and, respectively, $p_1 x_1 + \dots + p_n x_n$ in the elliptic case) and its differential $\diff{\varphi}$.
If by contradiction $\gamma_\Ccal (\Ccal) \subseteq L$, we have that $\diff{\varphi}$ is constantly zero on $\Ccal$, and this means that $\restr{\varphi}{\Ccal}$ is constant. 
Hence $\Ccal$ is contained in the 
algebraic coset $\Hvar{p}[c']$, for some $c' \in G$, which is in contradiction with the assumption that $\Ccal$ is not contained into any proper coset.
\medskip

Now, let $P \in \Ccal \cap W$ be a point such that $\Ccal$ is tangent to $W$ at $P$. 
Then $\gamma_\Ccal(P)$ is in $L \cap \gamma_\Ccal (\Ccal)$ and the latter is a finite set, as $\gamma_\Ccal (\Ccal)$ is not contained in $L$.
We have two possibilities:

\begin{enumerate}

\item $\gamma_\Ccal (\Ccal)$ is a point: this means that $\gamma_\Ccal (\Ccal)$ is exactly $\{\gamma_\Ccal (P)\}$ and therefore $\gamma_\Ccal (\Ccal)$ is contained in $L$;
as already remarked, this cannot happen under our assumptions.

\item $\gamma_\Ccal (\Ccal)$ is a curve: observe that in this case there is an upper bound on the cardinality $\abs{L \cap \gamma_\Ccal (\Ccal)}$ given by Bezout's Theorem;
such an upper bound does not depend on $L$, and is exactly the degree $d$ of $\gamma_\Ccal (\Ccal)$ with respect to the line bundle $O(1)$ (and the following holds: any hyperplane intersects $\gamma_\Ccal (\Ccal)$ in at most $d$ points and a generic hyperplane intersects $\gamma_\Ccal (\Ccal)$ in exactly $d$ points).

As $L$ is defined over $K$ (since by assumption $\End(G) \subseteq K$) and $\gamma_\Ccal (\Ccal)$ is defined over $K$, this means that:
\[
[K(\gamma_\Ccal (P)):K] \le \abs{L \cap \gamma_\Ccal (\Ccal)} \le d.
\]
It follows that $[K(P):K] \le d \cdot \deg{\gamma_\Ccal}$, where $\deg{\gamma_\Ccal}$ is the degree of the morphism of curves given by $\gamma_\Ccal: \Ccal \rightarrow \gamma_\Ccal (\Ccal)$.
\end{enumerate}
\end{proof}

\begin{remark}
    In the previous proposition, we actually prove something stronger: namely that there exists a constant $\tilde{\delta}$, depending only on the curve $\mathcal{C}$, such that for a fixed algebraic subgroup $H$ of codimension 1, there are at most $\tilde{\delta}$ points of singular intersection between the curve and the translates of $H$.
    In particular, this implies that the cosets $W$ that have a point of singular intersection with the curve $\mathcal{C}$ are all defined over a finite extension of $K$.
\end{remark}

By Northcott's Theorem, to prove Theorem \ref{thm:main} for $\Ccal$ not contained in any proper weakly special subvariety we only need to have a height bound for the points $P \in \Ccal$ where $\Ccal$ intersects tangentially a proper subgroup $H$.
Such bounds are already known by previous work of Bombieri, Masser and Zannier \cite{BMZ99} for $G=\Gm$ and Viada \cite{Viada2003} for $G=E$ an elliptic curve, who proved that the set
of points $P \in \Ccal$ such that there exists a one-codimensional special subvariety $H$ with $P \in \Ccal \cap H$ is a set of bounded height.

We finally remark that, differently from the usual proofs in unlikely intersections 
{(e.g. \cite{BMZ99, Viada2003, PZ08, HabeggerPila})}, the bound on the degree given by Proposition \ref{prop:degree} is independent of the bound on the height.
This is similar to what Marché and Maurin get in the special case of $\Gcal = \Gm^2$ in \cite{Maurin23}.

\section{Proof in the general case}
\label{sec:general_case}

The proof of the result under the more general assumption that the curve is not contained in a special subvariety is more difficult to handle; indeed, we do not have that the height of the points in our set is automatically bounded applying the results of Bombieri, Masser and Zannier \cite{BMZ99} for $G=\Gm$ and Viada \cite{Viada2003} for $G=E$ an elliptic curve. 
Since we already proved the result in the case when the curve $\Ccal$ is not contained in any weakly special subvariety, we can now assume that this is the case. 
We will now translate our problem (up to isogenies) in an analogous one: namely, we will show that there is a 1-1 correspondence between points of $\Ccal$ which intersect tangentially a special subvariety and points of another curve $\Ccal'$, lying in some $G^d$ with $d\le n$ and not contained in a weakly special subvariety, which intersect tangentially a weakly special subvariety of $G^d$ of a prescribed type (namely the $\vect{\xi}$-special subvarieties mentioned
in the Introduction).
Using this reduction, Theorem \ref{thm:main} will be a consequence of the (more general) Theorem \ref{thm:cspecial}, which is analogous to \cite[Theorem 1.6]{Remond2003},\cite[Conjecture 5.2]{Pink2005} and \cite[Théorème 1.6]{Maurin08} (see also \cite{Bombieri_Masser_Zannier2006}). 

\medskip

We recall our setting from Theorem \ref{thm:main}.
Let $G=\Gm$ or $E$ for some elliptic curve $E$ defined over a number field $K$; let $n$ be a positive integer, $\Gcal = G^n$ and let $\Ccal \subseteq \Gcal$ be a curve defined over $K$ and not contained in a proper special subvariety.
We want to prove that the set $\Ccal^{[1, sing]}$ of points of $\Ccal$ where $\Ccal$ intersects tangentially a proper special subvariety of $\Gcal$ is a finite set.

In the last section, we proved this under the assumption that $\Ccal$ was contained in no weakly special subvariety.
Now, we consider the case in which $\Ccal$ is actually contained in a proper weakly special subvariety $W$ (not special), which we suppose to have minimal dimension.
As shown in Section \ref{sec:structure}, up to an isogeny of $\Gcal$ we can assume that $W = G^{d} \times \{\xi_1\}\times\dots \times\{\xi_k\}$ with $d+k=n$.

Since $\Ccal$ is contained in $W$ and $\Ccal$ is defined over $K$, then also $\xi_1, \dots, \xi_k$ will be defined over $K$; moreover, these points will be (multiplicatively or linearly) independent in $G^k$, since otherwise $\Ccal$ would be contained in a proper special subvariety of $\Gcal$. We put $\vect{\xi} = (\xi_1, \dots, \xi_k)$.

If $\Ccal^{[1, sing]}$ is empty, there is nothing to prove; therefore, we assume that this is not the case and fix $P \in \Ccal^{[1, sing]}$.

We now wish to restrict our attention to the first $d$ components (with elements $\vect{y} = (y_1,\dots,y_{d})$) and consider the following injection:
\[
\begin{matrix}
    \iota: & G^{d} &\rightarrow & G^n\\
    &\vect{y} & \mapsto &(y_1,\dots,y_{d}, \xi_1,\dots,\xi_k) 
\end{matrix}
\]

Notice that $\iota(G^{d})$ is exactly $W$, and since we are assuming that $\mathcal C\subseteq W$, we can consider $\iota^{-1}(\Ccal)$; this is again a curve that we will denote by $\mathcal{C}'$.
\medskip

Using the minimality of $W$, we can show that $\mathcal{C'}$ is not contained in any proper weakly special subvariety of $G^{d}$. Indeed, if we assume by contradiction that $\iota^{-1}(\Ccal) \subseteq W'$ for $W'$ a proper weakly special subvariety of $G^{d}$, the image $\iota(W')$ would be again a weakly special subvariety strictly contained in $W$, contradicting the minimality of the dimension of $W$.
\medskip

We will be showing that $\iota$ induces a one-to-one correspondence between points in $\mathcal C \cap H$ where $H$ is a special subvariety of $G^n$ and points lying in the intersection between $\mathcal C'$ and a  weakly-special subvariety of $G^d$ of a particular shape.
For $\vect{a} \in (\mathrm{End}(G))^h$, we denote by $\varphi_{\vect{a},h}: G^h \rightarrow G$ the map defined by $\varphi_{\vect{a},h}(x_1, \dots, x_h)=x_1^{a_1}\cdots x_h^{a_h}$ if $G=\Gm$ and $\varphi_{\vect{a},h}(x_1, \dots, x_h)=a_1 x_1+ \cdots + a_h x_h$ if $G=E$.\\
Let $P \in \mathcal C \cap H_{\vect{a}}$, where $H_{\vect{a}}$ is the algebraic subgroup of $G^n$ defined by $\varphi_{\vect{a}, n}=O$, and $O$ is the identity element of $G$; since $\mathcal C \subseteq W$, in particular $P=(x_1, \dots, x_d, \xi_1, \dots, \xi_k)$. 
Then, $P\in \mathcal C \cap H_{\vect{a}}$ if and only if 
\begin{equation} \label{eq:xi_shape}
\varphi_{\vect{a}',d}(x_1, \dots, x_d) =\varphi_{-\vect{a}'',k}(\vect{\xi}), 
\end{equation}
with $\vect{a}' = (a_1, \dots, a_{d})$ and $\vect{a}''=(a_{d+1}, \dots, a_n)$. 
This means that the point $\iota^{-1}(P)=(x_1, \dots, x_k)$ lies in the intersection between $\mathcal C'$ and the subvariety of $G^d$ cut by \eqref{eq:xi_shape} (which is nothing else than $\iota^{-1}(H_{\vect{a}})$). 
Notice that this subvariety is the union of weakly-special subvarieties of $G^d$ of some particular form, since $\vect{\xi}$ is fixed.

\medskip
This characterization leads us to the following definition.

\begin{defn}\label{def:cspecial}
Let $k,d$ be two positive natural numbers and let $\vect{\xi}\in G(\C)^k$ be a $k$-tuple of independent nonzero elements $(\xi_1, \dots,\xi_k)$ of $G$. 
A $\vect{\xi}$-\emph{special hypersurface} of $G^{d}$ is an irreducible component of the zero locus of an equation of shape
\[
\varphi_{\vect{a}',d}(\vect{y}) = \varphi_{-\vect{a}'',k}(\vect{\xi}),
\]
with $\vect{a}'$ and $\vect{a}''$ tuples of endomorphisms such that $\vect{a}'$ is not the zero vector.
\end{defn}

Notice that $\vect{\xi}$-special hypersurfaces are weakly special subvarieties of $G^d$ of codimension one, and we showed that $P \in \mathcal{C} \cap H$ with $H$ a proper special subvariety of $G^n$ if and only if $\iota^{-1}(P)$ lies in the intersection of $\mathcal C'$ and a $\vect{\xi}$-\textit{special hypersurface} of $G^{d}$.

Next we will show that, if the first intersection is singular, then also the second one is singular.
This is quite intuitive, as essentially we are just restricting ourselves to the weakly special subvariety $\mathcal{W}$, in which $\mathcal{C}$ is contained.
More specifically, putting all together we have the following result.

\begin{lemma}\label{lemma:ctangent}
Let $\mathcal C \subset G^n$ be an irreducible curve of the form $\mathcal C' \times \{\xi_1\} \times \cdots \times \{\xi_k\}$, where $\mathcal C' \subset G^d$ is not contained in any weakly special subvariety and $\xi_1, \ldots, \xi_k \in G$ are independent. 
Let $H$ be a one-codimensional special subvariety of $G^n$ and let $P \in \Ccal \cap H$ be such that $\Ccal$ is tangent to $H$ at $P$. 
Then, there exists a  $\vect{\xi}$-special hypersurface of $G^d$, which we denote by $U$, such that $\mathcal C'$ is tangent to $U$ at $\iota^{-1}(P)$.
\end{lemma}

\begin{proof}
Let us denote by $W= G^d \times  \{\xi_1\}\times \cdots \times \{\xi_k\}$ the coset which contains $\Ccal$ by assumption. 
Let us first prove that $T_P(W \cap H)=T_PW \cap T_PH$. 
Notice that $W$ and $H$ are both translates of algebraic subgroups;
if we denote by $W_0$ and $H_0$ the images of $W$ and $H$ by the translation that maps $P$ to the origin, the statement above amounts to saying that the Lie algebra of (the irreducible component of) $W_0 \cap H_0$ (containing the origin) coincides with the intersection of the Lie algebras of $W_0$ and $H_0$. 
This is true both for the complex Lie group $\Gm^n$ and the complex lie group obtained from powers of elliptic curves: their Lie algebra can be identified with some $\C^n$ and its closed complex subgroups (which are all algebraic) correspond to subspaces of $\C^n$ defined over $\Q$, for $\Gm^n$ and regular elliptic curves, or the CM field, for CM elliptic curves; such subgroups are closed by intersection.

\medskip 
We now have that $T_P \Ccal \subseteq T_P H$ by hypothesis and, since $\Ccal \subseteq W$, then $T_P(\Ccal) \subseteq T_P(W \cap H)$ by the above argument. 
Observe that $\iota^{-1}$ is an isomorphism between $W$ and $G^d$, therefore $T_{\iota^{-1}(P)}(\iota^{-1}(\Ccal))$ is contained in $T_{\iota^{-1}(P)}(\iota^{-1}(H))$. 

Let now $\vect{a} \in \End(G)^n$ be such that $H$ is an irreducible component of the kernel $H_{\vect{a}}$ of $\varphi_{\vect{a}, n}$.
Then, $\iota^{-1}(H)$ is contained in the subvariety of $G^d$ cut out by 
\[
 \varphi_{\vect{a}',d}(\vect{y}) = \varphi_{-\vect{a}'',k}(\vect{\xi}),
\]
where $\vect{a}' = (a_1, \dots, a_{d})$ and $\vect{a}''=(a_{d+1}, \dots, a_n)$, and it is not empty since $P \in W \cap H$. 
Hence, if we take as $U$ the unique irreducible component of $\iota^{-1}(H_{\vect{a}})$ containing $P$, this satisfies our claim.
\end{proof}

Thanks to Lemma \ref{lemma:invariant_isogeny} and Lemma \ref{lemma:ctangent}, Theorem \ref{thm:main} will now be a consequence of the following result.

\begin{thm}\label{thm:cspecial}

Let $k$ and $d$ be positive integers, let $\vect{\xi}=(\xi_1,\ldots,\xi_k)$ be a vector of independent non-zero elements of $G(K)$ with $K$ a number field and let $\Ccal \subseteq G^d$ be a curve defined over $K$ which is not contained in any proper weakly special subvariety. 
Then, there are only finitely many points $P \in \Ccal$ such that there is a $\vect{\xi}$-special hypersurface $U$ such that:

\begin{enumerate}

\item $P$ is in $\Ccal \cap U$;

\item the tangent space $T_P \Ccal$ of $\Ccal$ at $P$ is included in the tangent space $T_P U$ of $U$ at $P$ (namely, the curve $\Ccal$ is tangent to $U$ at $P$).

\end{enumerate}

\end{thm}

This theorem is analogous to \cite[Conjecture 5.2]{Pink2005}, \cite[Theorem 1.6]{Remond2003} \cite[Théorème 1.6]{Maurin08}.
More generally, one can prove that the general version of the Zilber-Pink conjecture \cite[Conjecture 5.1]{Pink2005} is equivalent to \cite[Conjecture 5.2]{Pink2005} (see \cite[Theorem 5.5]{Pink2005} and \cite[Proposition 4.2]{Remond2003}).

%
%
%

\section{A bound for the height}

To prove Theorem \ref{thm:cspecial}, we will need a bound for the height of the points of singular intersection.
In what follows, we will always assume that our varieties are embedded in a projective space of which they inherit the usual projective logarithmic Weil height. 
For example, we embed $\Gm^n$ into $\P^n$ via:
\[
(x_1,\ldots,x_n) \rightarrow [x_1:\cdots :x_n : 1].
\]

We will now show that, if a point $P$ lies in the intersection of a curve with a one-codimensional weakly special subvariety that is an irreducible component of a coset of the form $\Hvar{p}[c]$, with $\vect{p} \in (\End(G))^n \setminus \{\vect{0}\}$ and $c \in G(\Qbar)$, and the intersection is tangent, then the height of $P$ is bounded in terms of the height of $\vect{p}$.

\begin{prop}\label{prop:tangentheight}
Let $G=\Gm$ or $G=E$ with $E$ an elliptic curve defined over a number field $K$.
Let $n$ be a positive integer and let $\Ccal \subseteq \Gcal=G^n$ be a curve again defined over $K$, not contained in any proper weakly special subvariety of $\Gcal$. 
Then, there exists a constant $\delta_2>0$, depending only on $\Ccal$ (and $\Gcal$) with the following property:
for every $\vect{p} \in (\End(G))^n \setminus \{\vect 0\}$ and $c \in G(\Qbar)$, for every irreducible component $H$ of $\Hvar{p}[c]$ and
for every $P \in \mathcal C(\Qbar)$ 
such that $\mathcal C$ and $H$ intersect at $P$ tangentially,
it holds that
\[
h(P) \le \delta_2(1+h(\vect{p})),
\]
where $h(\vect{p})$ is the projective height of $\vect{p} \in \P^{n-1}$.
\end{prop}

To prove the proposition, we first recall some properties of the height functions with respect to morphisms.

\begin{prop}[{\cite[Theorem B.2.5(a)]{Hindry2000}}]
\label{prop:heightUpperBound}
Let $\varphi:\P^n \dashrightarrow \P^m$ be a rational map of degree $d$ defined over $\Q$, so that $\varphi$ is given by a $(n + 1)$-tuple $\varphi = (f_0, \ldots , f_m)$ of homogeneous
polynomials of degree $d$. Let $Z \subseteq \P^n$ be the subset of common zeros of
the $f_i$'s. Then, $\varphi$ is defined on $\P^n \setminus Z$ and for all $P \in \P^n(\Qbar) \setminus Z$
\[
h(\varphi(P)) \leq d h(P) + O(1).
\]
\end{prop}

\begin{prop}[{\cite[Theorem 2.3]{Habegger2016}, from \cite[Theorem 1]{Silverman11}}]
\label{prop:heightLowerBound}
Let $X, Y$ be irreducible, quasi-affine varieties (where we suppose to fix a projective embedding and hence a height), defined over $\Qbar$, with $\dim X = \dim Y$.
Let $g: X \rightarrow Y$ be a dominant morphism.
Then, there exists a constant $\eta>0$, depending on $X$ and $Y$ and an open and dense subset $U$ of $X$ such that
\[
h(Q) \le \eta(1+h(g(Q)))
\]
for all $Q \in U(\Qbar)$.
Moreover, this estimate holds true with
\[
U = \{ P \in X : P \text{ is isolated in } \varphi^{-1}( \varphi(P) )\},
\]
which is Zariski open in $X$.
\end{prop}

In order to prove Proposition \ref{prop:tangentheight}, we recall the Gauss map introduced in Section \ref{sec:gauss}:

\[
\gamma_\Ccal: \Ccal^{sm} \rightarrow \P^{n-1} = \P T_e \mathcal G.
\]

Possibly removing the finitely many singular points, we will assume $\Ccal^{sm}=\Ccal$ and recall that, since $P \in \Ccal \cap_{sing}\Hvar{p}[c]$, by Corollary \ref{coroll:tangentplane} $\gamma_\Ccal(P)$ is contained in the hyperplane $L$ given by the equation:
\[
p_1 z_1+p_2 z_2+\cdots+p_n z_n=0,
\]
where $z_1,z_2,\ldots,z_n$ are the coordinates of $\P^{n-1}$. 
As in Section \ref{sec:no_coset}, if $\gamma_\Ccal (\Ccal)$ were contained in $L$, this would imply that $\Ccal$ is contained in a one-codimensional weakly special subvariety.
Since by assumption this is not the case, the set $\gamma_\Ccal^{-1}(\gamma_\Ccal(\Ccal) \cap L)$ is a finite set containing $P$ (as in Section \ref{sec:no_coset}, $\gamma_\Ccal$ is not constant, since $\gamma_\Ccal(\Ccal)$ is not contained in $L$).\\

By Proposition \ref{prop:heightLowerBound} we just need to bound the height of any element of $\gamma_\Ccal(\Ccal) \cap L$ (notice that $\Ccal$ and $\gamma_\Ccal(\Ccal)$ are already embedded in $\P^n$ and $\P^{n-1}$, respectively).
This amounts to proving the following.

\begin{prop}\label{prop:hyperplaneheight}

Let $s$ be a positive integer and $\Ccal' \subseteq \P^s$ be an irreducible curve defined over a number field $K$. Then, there is a constant $\delta'>0$, depending only on $\mathcal C'$, such that, for every hyperplane $L_{\vect{p}}$  given by the equation:
\[
p_0 z_0+p_1 z_1+\cdots+p_s z_s=0,
\]
with $\vect{p}=(p_0,p_1,\ldots,p_s) \in \Qbar^{s+1}\setminus \{\vect 0\}$, not containing $\mathcal C'$, and for every point $P \in \Ccal' \cap L$, one has
\[
h(P) \le \delta'(1+h(\vect{p})),
\]
where $h(\vect{p})$ is the projective height of $\vect{p} \in \P^s$.
\end{prop}

\begin{proof}
Recall that the Grassmannian $(\P^{s})^\vee$ parameterizing hyperplanes in $\P^s$, also known as its dual, is isomorphic to $\P^s$.
This isomorphism is induced by the choice of a projective basis: we choose the canonical one, and denote by $z_0, \dots z_s$ the coordinates (with respect to this basis) of the starting $\P^s$ and with $p_0,\ldots,p_s$ the coordinates in the Grassmanian $(\P^{s})^\vee$ with respect to the dual basis.
In other words, the point $[p_0 : \dots : p_s] \in (\P^{s})^\vee$ represents the hyperplane $L_{\vect{p}}$ of $\P^s$ given by the equation $p_0 z_0+p_1 z_1+\cdots+p_s z_s=0$.
\medskip

Let $V \subseteq (\P^{s})^\vee$ be the subset consisting of $\vect{p} \in (\P^{s})^\vee$ such that $L_{\vect{p}}$ does not contain $\Ccal'$; 
$V$ is an open subset of $(\P^{s})^\vee$ (since the requirement of containing a prescribed point is a closed condition and therefore $(\P^{s})^\vee \setminus V$ is an intersection of closed sets).

Our strategy consists in constructing some morphisms which behave well with respect to height in order to bound the height of $P$ from the height of $\vect{p}$.
\medskip

Let us consider $Sym^d(\Ccal') = (\Ccal')^d/S_d$ the $d$-th symmetric power of $\Ccal'$.
From \cite[Prop IV.1.5, p.180]{Knutson71}, $Sym^d(\Ccal')$ is a quasi-projective variety, so we can assume that it is embedded into a projective space, from which it inherits a height function. \footnote{A priori the constants that appear from now on will, and do in fact, depend on the choice of such projective embedding. However, the proof by Knutson construct an explicit projective embedding, and we may assume to have chosen that one.}

We can suppose that $\Ccal'$ is closed in $\P^s$ (this would just enlarge $\Ccal' \cap L_{\vect{p}}$). 
By Bézout's Theorem, there is a positive integer $d$ such that, for any $\vect{p} \in V$, we have that $L_{\vect{p}}$ intersects $\Ccal'$ in precisely $d$ points (counted with multiplicity). This gives rise to the following morphism, defined over $K$:
\[
\begin{matrix}
\phi_1: &V &\rightarrow &Sym^d(\Ccal') \\
&(\vect{p}) & \mapsto &\Ccal' \cap_m L_{\vect{p}},
\end{matrix}
\]

where $\Ccal' \cap_m L_{\vect{p}}$ denotes the schematic intersection of $\Ccal'$ and $L_{\vect{p}}$, taking into account multiplicities.
Since $Sym^d(\Ccal')=(\Ccal')^d/S_d$ is the quotient of $(\Ccal')^d$ by the finite group $S_d$, the quotient map defines a finite morphism, defined over $K$:
\[
\begin{matrix}
\phi_2:&(\Ccal')^d &\rightarrow  &Sym^d(\Ccal')\;\;\\
&(P_1,\ldots,P_d) &\mapsto &(P_1,\ldots,P_d)_{no}
\end{matrix}
\]

where the $d$-uple is ordered in the LHS and non-ordered in the RHS (we use the subscript $\text{}_{no}$ to denote this).
Finally, we consider the projection on the first coordinate:
\[\begin{matrix}
\phi_3: &(\Ccal')^d  &\rightarrow &\Ccal' \\
&(P_1,\ldots,P_d) &\mapsto& P_1
\end{matrix}\]

To summarize, we have the following diagram:
\[\begin{tikzcd}[column sep=-0.5em]
	& {\,(\Ccal')^d\,} && {\hspace{-0.5 em} V\subseteq(\P^{s})^\vee} \\
	{\quad\Ccal'\quad} && {\hspace{-1 em}Sym^d(\Ccal') \hspace{-1em}}
	\arrow["{\phi_3}"', from=1-2, to=2-1]
	\arrow["{\phi_2}", from=1-2, to=2-3]
	\arrow["{\phi_1}"', from=1-4, to=2-3]
\end{tikzcd}\]

We will now see that these maps behave well with respect to height.

Since $V$ is an open set of $(\P^{s})^\vee$, the map $\phi_1$ extends (formally) to a rational map between projective spaces.

Therefore, by Proposition \ref{prop:heightUpperBound}, there exists a constant $\eta_1>0$ such that, for all $\vect{p} \in V(\Qbar)$, we have
\[
h(\phi_1(\vect{p}))\leq \eta_1(1+h(\vect{p})).
\]
If we consider $(\Ccal')^d \subseteq (\P^s)^d \subseteq \P^{(s+1)^d-1}$, then the map $\phi_3$ can also be extended to a rational map between projective spaces such that it is defined on $(\Ccal')^d$, and in a similar way we get that there exists a constant $\eta_3>0$ such that, for all $(P_1, \dots, P_d) \in (\Ccal')^d (\Qbar)$,
\[
h(P_1) = h(\phi_3((P_1,\ldots,P_{d})))\leq \eta_3(1+h((P_1,\ldots,P_{d})).
\]

\medskip
Let us now focus on $\phi_2$; this is the quotient map by a finite group, so
$(\Ccal')^d$ and $Sym^d(\Ccal')$ have the same dimension, $\phi_2$ is surjective, and all the fibers are finite; 
in particular, all the points $P \in (\Ccal')^d$ are isolated in their fiber.
We can then apply Propostition \ref{prop:heightLowerBound} with $U=(\Ccal')^d$, which gives us a constant $\eta_2>0$ such that
\[
h((P_1,\ldots,P_{d}))\leq\eta_2(1+h(\phi_2((P_1,\ldots,P_{d})))=\eta_2(1+h((P_1,\ldots,P_{d})_{no})).
\]
Let now be $P \in \Ccal' \cap L_{\vect{p}}$. 
We can choose points $Q_1,\ldots,Q_{d-1}$ in $\Ccal'$ such that the element $(P,Q_1,\ldots,Q_{d-1})_{no}=\Ccal' \cap_m L_{\vect{p}}$. We have:
\begin{align*}
\phi_1(\vect{p})&=(P,Q_1,\ldots,Q_{d-1})_{no} \\
\phi_2(P,Q_1,\ldots,Q_{d-1})&=(P,Q_1,\ldots,Q_{d-1})_{no} \\
\phi_3(P,Q_1,\ldots,Q_{d-1})&=P. 
\end{align*}
We thus have
\begin{align*}
h(P)& \le \eta_3 (1+h((P, Q_1, \ldots, Q_{d-1})))\\
& \le\eta_3(1 + \eta_2(1 + h((P, Q_1, \ldots, Q_{d-1})_{no})))\\
& \le \eta_3(1 + \eta_2(1 + \eta_1(1 + h(\vect{p}))))
\end{align*}
Putting $\delta' = (\eta_1 + 1)(\eta_2 +1)(\eta_3+1)$ we conclude.
\end{proof}

\begin{remark}
    An alternative proof of Proposition \ref{prop:hyperplaneheight}, under the additional condition that $\vect{p}=(p_0,p_1,\ldots,p_s) \in \Q^{s+1}$, has also been provided by the first named author in his PhD thesis \cite[\S 5.5]{BalliniTesi}, using Puiseux series.
    Such additional condition is satisfied whenever we are considering subgroups of $\Gm^n$ or of elliptic curves without complex multiplication.
\end{remark}


\section{Proof of Theorem \ref{thm:cspecial}} 

This section is dedicated to proving Theorem \ref{thm:cspecial}.
We recall the notation introduced in the statement.
Let $G=\Gm$ or $G=E$, for $E$ an elliptic curve defined over a number field $K$.
Let $k$ and $d$ be positive integers, let $\vect{\xi}=(\xi_1,\ldots,\xi_k)$ be a vector of independent non-zero elements of $G(K)$ and let $\Ccal \subseteq \Gcal = G^d$ be a curve defined over $K$ which is not contained in any proper weakly special subvariety. 
Fix a point $P \in \Ccal$ such that there is some $\vect{\xi}$-special hypersurface $U$ with $\Ccal$ tangent to $U$ at $P$.

Let us assume that $U$ is contained in the zero locus of
\[
\varphi_{\vect{p}_U,d}(\vect{y}) = \varphi_{-\vect{q}_U ,k}(\vect{\xi}),
\]
where $\vect{p}_U, \vect{q}_U$ are tuples of endomorphisms of $G$ such that not all of the $p_i$ are zero.
We want to show that we can assume $\vect{p}_U$ and $\vect{q}_U$ to be sufficiently small in an appropriate sense.

Let $\Gamma$ be the submodule of $\End(G)^{d+k}$ consisting of the tuples $(a_1,\ldots,a_d,b_1,\ldots,b_k)$ such that:
\[
\varphi_{\vect{a}, d}(P) = \varphi_{-\vect{b}, k}(\vect{\xi}),
\]
and let $\Gamma_{\textrm{sing}}$ be the submodule of $\Gamma$ consisting of the tuples $(a_1,\ldots,a_d,b_1,\ldots,b_k)$ such that the curve $\Ccal$ is tangent at $P$ to the coset cut by the equation
\[
\varphi_{\vect{a}, d}(\vect{y}) = \varphi_{-\vect{b}, k}(\vect{\xi}).
\]

Recall that, for a module $M$ over an integral domain $\mathcal{O}$ with field of fractions $K$, the \emph{rank} of $M$ is defined as the rank of a maximal free $\mathcal{O}$-submodule contained in $M$, or, equivalently, with the $K$-dimension of $M \otimes_\mathcal{O} K$.

Notice that the ranks of $\Gamma$ and $\Gamma_{\textrm{sing}}$ are at least one, since $P \in U$ and the intersection of the curve with $U$ at $P$ is tangent, so $(\vect{p}_U, \vect{q}_U) \in \Gamma_{\textrm{sing}} \subseteq \Gamma$.
There are then two possibilities:

\begin{enumerate}
\item If $\rk (\Gamma) \ge 2$, then there are finitely many possible values for $P$ as a consequence of Zilber-Pink-like statements for a curve in a power of $\Gm$, due to Maurin \cite[Théorème 1.6]{Maurin08}, or in a power of an elliptic curve, due to Viada \cite{Viada2003, Viada2008}, Rémond and Viada \cite[Théorème 1.5-1.6]{Remond2003} and Galateau \cite[Théorème 1.10]{Galateau2010}.

\item If, on the other hand, $\rk(\Gamma)=\rk(\Gamma_\textrm{sing})=1$,
let $r=[\End(G):\Z]$, so $r=1$ for $G=\Gm$ or $G=E$ for a regular elliptic curve, and $r=2$ for $G=E$ a CM elliptic curve.
We also define
\[
\widetilde{h}_{\textrm{aff}}(P):= \max{\{h(\xi_1),\ldots,h(\xi_k),h(y_1(P)),\ldots,h(y_d(P))\}}.
\]
By a consequence of a result of Masser \cite{Masser88} (see \cite[Lemmas 5.1 and 5.2]{BarroCapuano17} and the proof of \cite[Lemma 6.1]{BarroeroCM}),  
for an absolute constant $\eta_{d+k}>0$,
the module $\Gamma$ contains a non-zero element $(p_1,\ldots,p_d,q_1,\ldots,q_k)$ with
\begin{align} \label{eq:height_bound}
\max  \{\abs{p_1},\ldots &,\abs{p_d},\abs{q_1},\ldots,\abs{q_k}\} \le \\
&\le \eta_{d+k} [K(P):K]^{r(d+k)+1} (1+\widetilde{h}_{\textrm{aff}}(P))^{r(d+k)-1}. \nonumber
\end{align}
We will need a similar small element in $\Gamma_{sing}$. This is ensured by the following.
\begin{lemma}
Let $F$ be the field of fractions of $\End(G)$. Then $F^{d+k} \supseteq F\Gamma_{sing} \cap \Gamma = \Gamma_{sing}$.
In particular, the element $(p_1,\ldots,p_d,q_1,\ldots,q_k)$ above lies in $\Gamma_{sing}$.
\end{lemma}
\begin{proof}
Clearly, $F\Gamma_{sing} \cap \Gamma \supseteq \Gamma_{sing}$, so we only need to prove the reverse inclusion.
Recall that in Section \ref{sec:general_case} we defined the map $\iota$ by
\[\begin{matrix}
    \iota:& G^d & \to &G^d \times G^k \\
    & \vect{y} & \mapsto & (\vect{y}, \vect{\xi}).
\end{matrix}\]
We will, throughout this proof, tacitly apply Lemma \ref{lemma:ctangent}.

Let $(a_1,\ldots,a_d,b_1,\ldots,b_k) \in F \Gamma_{sing} \cap \Gamma$; then, $\iota(P)=(P, \vect{\xi})$ lies in the hypersurface $U_{\vect{a}, \vect{b} }$ of
$G^d \times G^k$ cut out by
$\varphi_{\vect{a}, d}(\vect{y}) = \varphi_{-\vect{b}, k}(\vect{z})$. 
Moreover, there exists $\mu \in F^{\times}$ such that the curve $\iota(\Ccal)$ is tangent to the hypersurface $U_{\mu \vect{a}, \mu \vect{b}}$ 
cut by $\varphi_{\mu \vect{a}, d}(\vect{y}) = \varphi_{-\mu\vect{b}, k}(\vect{z})$ in $G^d \times G^k$ at $(P, \vect{\xi})$.
We want to prove that, actually, $\iota(\Ccal)$ is tangent to $U_{\vect{a}, \vect{b}}$ at $(P, \vect{\xi})$.

Let $T=T_{(P, \vect{\xi})}(G^{d+k})$ be the tangent space in $\iota(P)$, which we identify with the tangent space at the origin, with coordinates $\partial{y_1}, \dots, \partial{y_d}, \partial{z_1}, \dots, \partial{z_k}$.
Now, the curve $\iota(\Ccal)$ is tangent to $U_{\vect{a}, \vect{b}}$ at $(P, \vect{\xi})$ if and only if we have that the line $T_{(P, \vect{\xi})}(\iota(\Ccal))$ is contained in the hyperplane $a_1 \partial{y_1} + \dots + a_d \partial{y_d} + b_1 \partial{z_1} + \dots + b_k \partial{z_k}=0$.
By assumption, $T_{(P, \vect{\xi})}\iota(\Ccal)$ is contained in $\mu a_1 \partial{y_1} + \dots + \mu a_d \partial{y_d} + \mu b_1 \partial{z_1} + \dots + \mu b_k \partial{z_k}=0$, and clearly these two are the same hyperplane, so our claim is proven.
\end{proof}

Notice that, by Proposition \ref{prop:degree}, we have an absolute upper bound on $[K(P):K]$; combining this with \eqref{eq:height_bound} we obtain:
\begin{equation} \label{eq:height_bound2}
\max{\{\abs{p_1},\ldots,\abs{p_d},\abs{q_1},\ldots,\abs{q_k}\}} \le \rho (1+h(P))^{r(m+n)-1},
\end{equation}
where $\rho>0$ is a constant depending only on $\Ccal$ and $\vect{\xi}$. 
By Proposition \ref{prop:tangentheight} we have 
\begin{equation} \label{eq:height_bound3}
h(P) \le \rho'(1+\log{\max{\{\abs{p_1},\ldots,\abs{p_d}\}}}),
\end{equation}
for a constant $\rho'>0$ depending only on $\Ccal$. 
Combining \eqref{eq:height_bound2} and \eqref{eq:height_bound3} we obtain an upper bound on $\abs{p_1},\ldots,\abs{p_n}$ and hence an upper bound on $h(P)$. 
Finally, together with the upper bound for $[K(P):K]$ given by Proposition \ref{prop:degree}, we conclude the proof of Theorem \ref{thm:cspecial} applying Northcott's theorem.

\end{enumerate}

\section{An application to divisibility sequences}
In this section we discuss an application of Theorem \ref{thm:split_semiab} to geometric divisibility sequences.

\medskip
Let $S$ be an irreducible curve. A \emph{divisibility sequence} on $S$ is a sequence of divisors $(D_n)_{n \geq1}$ of $S$ such that, for every $ m, n \in \N$, if $m \mid n$ then $D_m \mid D_n$, that is, $D_n - D_m$ is an effective divisor.
This notion is the function field analogue of an integral divisibility sequence (see \cite{BCZ2003} for results about the greatest common divisor of independent sequences of the form $a^n-1$).\\
Let $\pi:\mathcal{G} \to S$ be a group scheme with zero-section $O$.
As noted by Silverman in \cite{Silverman2005}, every section $P:S\to \mathcal{G}$ that is not identically torsion, i.e. $nP \neq O$ for all $n\geq 1$, has an associated divisibility sequence given by
\begin{equation}
\label{eqn:geom_seq}
D_{nP} := O^*(nP),
\end{equation}
i.e., $D_{nP}$ is the pull-back along the zero section of the divisor $nP$ on $\mathcal{G}$; such divisibility sequences are called \emph{geometric}.\\
A conjecture of Silverman, stated in \cite{Silverman2005}, predicts that, if $\mathcal{G}$ has relative dimension at least $2$, and the images of the multiples of $P$ are dense in the generic fiber, then
$D_{nP} = D_P$ for infinitely many $n>0$.
This question has been studied by Ailon and Rudnick \cite{Ailon_Rudnick} when $\mathcal G$ is the constant scheme $\Gm^2 \times S$, by
Silverman \cite{Silverman2004} and Ghioca, Hsia and Tucker \cite{Ghioca2018} when $\mathcal{G}$ is the product of two elliptic schemes, and by Barroero, Capuano and Turchet \cite{BCT2024} in the case of split semiabelian schemes.
The results proved in \cite{Ghioca2018} and \cite{BCT2024} are stronger than the one mentioned above: under the conditions predicted by Silverman's conjecture, the set of $n \in \N$ for which $D_{nP} = D_P$ is the complement of a finite union of arithmetic progressions.

Generalizing the setting above, if we have two sections $P$ and $Q$ of $\mathcal{G} \to S$, we can construct a sequence of divisors by replacing the zero section in \eqref{eqn:geom_seq} with $Q$, provided that $nP \neq Q$ for all $n\geq 1$.
We then get
\begin{equation}
\label{D_nP_Q_eqation}
D_{nP, Q} := Q^*(nP) = \pi_*(nP \cap Q),
\end{equation}
where, by slight abuse of notation, we are identifying $nP$ and $Q$ with their image, and the intersection is intended to be the scheme-theoretic intersection as zero-cycles.\\
These sequences have also been studied by Ghioca, Hsia and Tucker \cite{Ghioca2018} and Barroero, Capuano and Turchet \cite{BCT2024},
for $\mathcal{G}$ the product of two elliptic schemes and a split semiabelian scheme respectively (see also \cite{Ostafe_16} for some related results in the case $\Gcal= \Gm^2 \times S$).
They proved that, under the same conditions on $P$ as the above mentioned conjecture of Silverman's, $D_{nP, Q}$ is bounded by a divisor independent of $n$.

\medskip

When the relative dimension of $\mathcal{G}$ is 1, then there is no expectation that $D_{nP} = D_P$ infinitely often.
However, in this case it is interesting to study when $D_{nP}$ is reduced, that is, it is of the form $D_{nP} = \sum_i t_i$ for $t_i \in \mathcal{C}$ distinct points (equivalently, each non-zero coefficient of $D_{nP}$ is equal to 1). 
This condition is an analogue to the fact that an integer is square-free. From a geometric point of view, this amounts to the fact that the section $nP$ intersects transversally the zero section at every point.

In \cite{UU21}, the authors prove that, for $\mathcal{G}=\mathcal{E}$ an elliptic scheme and $P$ non identically torsion, $D_{nP}$ is reduced for all $n$ outside a finite union of arithmetic progressions. Using March\'e and Maurin result \cite[Theorem 1.4]{Maurin23} one can prove that the same holds when $\mathcal G$ is a constant scheme over a curve with $\Gm$ fibers.

\begin{thm}
\label{thm:toric_divisibility_seq}
Let $G= \Gm$, let $S$ be a smooth curve defined over $\Qbar$ and consider the constant scheme $\Gcal := G \times S$ with structural morphism $\pi: \Gcal \to S$.
For any section $P: S \to \Gcal$
there is a finite set of integers $M = \{m_1 , \dots , m_k \}$ such that $O$ and $nP$ intersect transversally if and only if $n$ is not divisible by any element of $M$.
\end{thm}
This is analogous to the constant case of \cite[Theorem 1.4]{UU21}, and the same proof applies (see also \cite[Paragraph 1.2]{UU20}).

\begin{example}
The conclusion of Theorem \ref{thm:toric_divisibility_seq} is very natural, and in some cases can be proved in elementary terms, as this example shows. 

Let $f(t)\in \C[t]$ be a nonconstant polynomial; then, $f(t)\in \Gm(\C(t))$ is a rational function. For every $n \ge 1$ we can then define the divisor $D_{f^n}=\sigma_{f^n}^*(1_{\Gm})$, where $\sigma_{f^n}$ denotes the section of $\mathcal G$ corresponding to the rational function $f^n$ and $1_{\Gm}$ is the identity section corresponding to the point $1 \in \Gm(\C(t))$.
In this setting, the divisor $D_{f^n}$ is reduced if and only if the polynomial $f(t)^n-1$ has only simple roots.

We notice that, for every $n\ge 1$, we can factorize $f(t)^n-1$ as
$$ f(t)^n-1= \prod_{j=0}^{n-1} (f(t)-\zeta_n^j), $$
with $\zeta_n=e^{2\pi i/n}$. 
It is easy to see that for every $j \neq j'$, the factors $f(t)-\zeta_n^j$ and $f(t)-\zeta_n^{j'}$ do not have common roots, hence $f(t)^n-1$ has a multiple root $t_0 \in \C$ if and only if $t_0$ is a multiple root of one of the factors. 
On the other hand, if $t_0$ is a multiple root of some $f(t)-\zeta_n^j$, then $f'(t_0)=0$. 
Consequently, the multiple roots of the polynomials of the form $f(t)^n-1$ are exactly the $t_0\in \C$ such that $f'(t_0)=0$ and there exists $n\ge 1$ and $0 \le j \le n-1$ such that $f(t_0)=e^{2\pi j/n}$. 
Notice that these elements are $\le \deg f-1$. For each of these elements, denote by $m_{t_0}$ the smallest $m \ge 1$ such that $f'(t_0)=0$ and $f(t_0)=e^{2\pi j/m_{t_0}}$ with $0 \le j \le m_{t_0}$; then, $t_0$ will be a multiple root of $f(t)^n-1$ for every $n$ which is a positive multiple of $m_{t_0}$. 
This means that $D_{nP}$ is reduced if and only if $n$ is not divisible by any of these $m_{t_0}$, as implied by Theorem \ref{thm:toric_divisibility_seq}.
\end{example}

Using Theorem \ref{thm:split_semiab} we can prove that a similar statement holds for $D_{nP, Q}$, as done in \cite{Ottolini} where the author considers the case of sections in non isotrivial elliptic schemes. 

Before stating and proving our result for $D_{nP, Q}$, we will need a preliminary result.

\begin{thm}
\label{Application_thm_1}
Let $G= \Gm$ or $G=E$ for an elliptic curve $E$ defined over a number field $K$, let $S$ be a smooth curve and consider the constant scheme $\Gcal := G \times S$ with structural morphism $\pi: \Gcal \to S$. 
Let $P, Q$ be sections $S \to \Gcal$.
If there exist infinitely many $s \in S(\overline{\Q})$ for which there exists $m_s \in \Z$ such that $m_s P(s) = Q(s)$ tangentially, then, for some $m \in \Z$, we have that $mP=Q$ identically on $S$.
\end{thm}
\begin{proof}
Consider $(s_n)_n$ an infinite sequence of pairwise distinct $s_n \in S(\overline{\Q})$ for which there exists $m_{s_n} \in \Z$ such that
\begin{equation}
\label{tangent_intersection_section_eqn}
m_{s_n} P(s_n) = Q(s_n) \text{  tangentially}.
\end{equation}

We have two cases: either the $m_{s_n}$ are bounded, or they are unbounded.
We will treat these two cases separately.

If the $m_{s_n}$ are bounded, then, up to taking a subsequence, we have that for a fixed $m$ there exist infinitely many $s \in S(\overline{\Q})$ where $m P(s) = Q(s)$, from which the claim follows.

The case in which the $m_{s_n}$ are unbounded is left to prove, so we will be assuming to be in that case.
From Theorem \ref{thm:main} it follows that, for some $a, b \in \Z$ not both zero,
\begin{equation}
\label{relation_identical_section_eqn}
a P = b Q
\end{equation}
identically on $S$.
Combining this with \eqref{tangent_intersection_section_eqn}, we find that $P(s_n)$ is torsion for infinitely many $n$, since we have
\begin{equation}
\label{p_section_is_torsion_eqn}
O(s_n)=a P(s_n) - b Q(s_n) = a P(s_n) - b m_n P(s_n) = (a-b m_n)P(s_n)
\end{equation}
and $(a- b m_n)$ is non-zero infinitely often, as $m_n$ is unbounded.
Moreover, since \eqref{tangent_intersection_section_eqn} is true tangentially and \eqref{relation_identical_section_eqn} is true identically, we get that \eqref{p_section_is_torsion_eqn} also holds tangentially.
Again by Theorem \ref{thm:main}, $P$ must be identically torsion, and so the same is true for $Q$, by \eqref{relation_identical_section_eqn}.
Since by assumption there is a certain $s \in S(\Qbar)$ and $m \in \Z$ such that $m P(s) = Q$, such relation must hold for all $s \in S(\Qbar)$, and hence identically on $S$.
\end{proof}
\begin{remark}
Notice that in the case $G=\Gm$, the result by Marché and Maurin \cite{Maurin23} is sufficient to prove the previous theorem.
\end{remark}

\begin{corollary}
\label{coroll_finite_set}
    Let $G, S, \Gcal, \pi$ be as in the previous theorem, and let $P, Q$ be sections $S \to \Gcal$ such that no multiple of $P$ is identically equal to $Q$.
    Let $S'$ be the set of $s \in S$ such that, for some $n \in \N$, $2s \leq D_{nP, Q}$.
    Then, $S'$ is a finite set.
\end{corollary}

Before stating and proving the next theorem, we introduce the following notation.
For an effective divisor $D$ of a curve $S$ we call the radical of $D$, denoted as $\Rad(D)$, the largest reduced effective divisor $E$ of $S$ such that $E \leq D$.
More concretely, if $D = \sum_{i \in I} a_i t_i$ with $t_i$ points of $S$ and $a_i > 0$ for all $i \in I$, then
\[
 \Rad(D) = \sum_{i \in I} t_i.
\]
We call the non-reduced part of $D$, denoted as $D^{\mathrm{nonRed}}$, the divisor $D - \Rad(D)$. 
Note that this is always an effective divisor, although possibly the zero one, in case $D$ is reduced. 

\begin{thm}
\label{thm:elliptic_div}
Let $G= \Gm$ or $G=E$ for an elliptic curve $E$ defined over a number field $K$, let $S$ be a smooth curve and consider the constant scheme $\Gcal := G \times S$ with structural morphism $\pi: \Gcal \to S$. 
Let $P, Q$ be sections $S \to \Gcal$ such that no multiple of $P$ is identically equal to $Q$.
Let $D_{nP, Q}$ be as in \eqref{D_nP_Q_eqation}.
Then:
\begin{enumerate}
\item there exists a divisor $D$ of $S$, independent of $n$, such that $D_{nP, Q}^{\mathrm{nonRed}} \leq D$;
\item the set 
\[
\Xi = \{ n \in \N \st D_{nP, Q} \text{ is not reduced} \}
\]
is the union of a finite set and a finite union of arithmetic progressions.
\end{enumerate}
\end{thm}

\begin{proof}
For $s \in \mathcal{C}$ let $\Xi_s = \{ n \in \N \st 2s \leq D_{nP, Q} \}$, that is, the set of $n\in \N$ such that $nP$ and $Q$ intersect tangentially in $s$.
Clearly, $\Xi = \bigcup_{s \in \mathcal{C}} \Xi_s$.
Moreover, $\Xi_s$ is nonempty exactly whenever $s$ lies in the finite set $S'$ of Corollary \ref{coroll_finite_set}. 
In particular, the support of $D_{nP, Q}^{\mathrm{nonRed}}$ is contained in such $S'$.
To get part (1), it is now enough to apply \cite[Lemma 4.5]{Ghioca2018} (see also \cite[Lemma 4.4]{BCT2024}).

For part (2), we need to show that, for each $s \in S$, we have that $\Xi_s$ is either empty, a singleton, or a finite union of arithmetic progressions.
Suppose that $\Xi_s$ is nonempty, and let $n_0$ be its smallest element. For any other $n \in \Xi_s$ we have that both $n_0 P$ and $n P$ intersect $Q$ tangentially over $s$. Hence, $(n - n_0) P$ intersects the zero section tangentially over $s$. 
By \cite[Thm 1.4]{UU21} in the elliptic case, and Theorem \ref{thm:toric_divisibility_seq}, we can deduce that $\Xi_s - n_0$ is a finite union of arithmetic progressions starting from 0, so $\Xi_s$ itself is a finite union of arithmetic progressions starting from $n_0$.
\end{proof}

\begin{remark}
    The result of part (1) of the previous theorem also holds in the case of non-isotrivial elliptic scheme, with the same proof, although it is not stated explicitly in \cite{Ottolini}.
\end{remark}

\begin{remark}
We point out that the result of part (2) of Theorem \ref{thm:elliptic_div} is weaker than what Ulmer and Urz{\'u}a get in the special case $Q = O$ (\cite[Thm 1.4]{UU21}) (see also Theorem \ref{thm:toric_divisibility_seq}). 
There, the authors prove that $\{ n \in \N \st D_{nP} \text{ is not reduced} \}$ is the finite union of arithmetic progressions \emph{starting from 0}, implying that if $D_{P}$ is reduced, then $D_{nP}$ is reduced infinitely often.
In our case, the arithmetic progressions may not start from 0, so Theorem \ref{thm:elliptic_div} does not provide an easy method to check whether $D_{nP, Q}$ is reduced infinitely often.
\end{remark}

\section*{Acknowledgements}

The authors thanks Luca Ferrigno, Vincenzo Mantova, David Masser, Jonathan Pila, Amos Turchet and Umberto Zannier for very fruitful conversations on the topic.
The first-named author was supported by the PhD program in Mathematics at the University of Oxford.
The second-named author was supported by the PRIN 2022 project \textit{“Semi-abelian varieties, Galois representations and related Diophantine problems”}.
The third-name author was supported by the Ph.D. program in Mathematics at the University of Rome ``Tor Vergata'', and by the SNSF project \textit{``Atypical and typical intersections in semiabelian varieties''} (grant number 10006332).
The second and third authors are parts of the INdAM group GNSAGA.

\bibliographystyle{alpha-abbr}
\bibliography{biblio} 

\end{document}